\documentclass[runningheads, envcountsame, a4paper]{llncs}

\usepackage{amssymb,amsmath}
\usepackage{url}
\usepackage[colorlinks=true,citecolor=black,linkcolor=black,urlcolor=blue]{hyperref}
\usepackage{cleveref}

\crefname{theorem}{Theorem}{Theorems}
\crefname{corollary}{Corollary}{Corollaries}
\crefname{proposition}{Proposition}{Propositions}
\crefname{claim}{claim}{claims}
\Crefname{claim}{Claim}{Claims}
\crefname{conjecture}{Conjecture}{Conjectures}
\crefname{remark}{Remark}{Remarks}

\def\abs#1{\left| #1 \right|}
\def\paren#1{\left( #1 \right)}
\def\acc#1{\left\{ #1 \right\}}
\def\floor#1{\left\lfloor #1 \right\rfloor}

\def\cro#1{\left[ #1 \right]}

\renewcommand{\le}{\leqslant}
\renewcommand{\ge}{\geqslant}

\begin{document}
\title{Nice formulas, $xyx$-formulas, and palindrome patterns\thanks{This work was partially supported by the ANR project CoCoGro (ANR-16-CE40-0005)}}

\titlerunning{Nice formulas and palindrome patterns}  
\toctitle{Nice formulas and palindrome patterns}

\author{Pascal Ochem\inst{1} \and Matthieu Rosenfeld\inst{2}}

\authorrunning{Pascal Ochem and Matthieu Rosenfeld} 

\tocauthor{Pascal~Ochem and Matthieu~Rosenfeld}
\institute{LIRMM, CNRS, Universit\'e de Montpellier, France.\\
\email{ochem@lirmm.fr}
\and LIP, ENS de Lyon, CNRS, UCBL, Universit\'e de Lyon, France.\\
\email{matthieu.rosenfeld@ens-lyon.fr}}

\maketitle
\setcounter{footnote}{0}
\begin{abstract}
We characterize the formulas that are avoided by every $\alpha$-free word for some $\alpha>1$. We study the avoidability index of formulas whose fragments are of the form $XYX$.
The largest avoidability index of an avoidable palindrome pattern is known to be at least $4$ and at most $16$. We make progress toward the conjecture that every avoidable palindrome pattern is $4$-avoidable.

\end{abstract}

\section{Introduction}\label{sec:intro}
A \emph{pattern} $p$ is a non-empty finite word over an alphabet $\Delta=\acc{A,B,C,\dots}$ of capital letters called \emph{variables}.
An \emph{occurrence} of $p$ in a word $w$ is a non-erasing morphism $h:\Delta^*\to\Sigma^*$ such that $h(p)$ is a factor of $w$ (a morphism is \emph{non-erasing} if the image of every letter is non-empty).
The \emph{avoidability index} $\lambda(p)$ of a pattern $p$ is the size of the smallest alphabet $\Sigma$ such that there exists an infinite word over $\Sigma$ containing no occurrence of $p$.
Since there is no risk of confusion, $\lambda(p)$ will be simply called the index of $p$.

A variable that appears only once in a pattern is said to be \emph{isolated}.
Following Cassaigne~\cite{Cassaigne1994}, we associate a pattern $p$ with the \emph{formula} $f$ obtained by replacing every isolated variable in $p$ by a dot.
The factors between the dots are called \emph{fragments}.

An \emph{occurrence} of a formula $f$ in a word $w$ is a non-erasing morphism $h:\Delta^*\to\Sigma^*$ such that the $h$-image of every fragment of $f$ is a factor of $w$.
As for patterns, the index $\lambda(f)$ of a formula $f$ is the size of the smallest alphabet allowing the existence of an infinite word containing no occurrence of $f$.
Clearly, if a formula $f$ is associated with a pattern $p$, every word avoiding $f$ also avoids $p$, so $\lambda(p)\le\lambda(f)$.
Recall that an infinite word is \emph{recurrent} if every finite factor appears infinitely many times and that any infinite factorial language contains a recurrent word~\cite[Proposition 5.1.13]{Fogg}.
If there exists an infinite word over $\Sigma$ avoiding $p$, then there exists an infinite recurrent word over $\Sigma$ avoiding $p$.
This recurrent word also avoids $f$, so that $\lambda(p)=\lambda(f)$. Without loss of generality, a formula is such that no variable is isolated and no fragment is a factor of another fragment.

Let us define the types of formulas we consider in this paper.
A pattern is \emph{doubled} if it contains every variable at least twice. Thus it is a formula with only one pattern.
A formula $f$ is \emph{nice} if for every variable $X$ of $f$, there exists a fragment of $f$ that contains $X$ at least twice. Notice that a doubled pattern is a nice pattern.
A formula is an \emph{$xyx$-formula} if every fragment is of the form $XYX$, i.e., the fragment has length $3$ and the first and third variable are the same.
A formula is \emph{hybrid} if every fragment has length 2 or is of the form $XYX$. Thus, an $xyx$-formula is a hybrid formula.

In \Cref{sec:nice}, we consider the avoidance of nice formulas.
In \Cref{sec:charac}, we find some formulas $f$ such that every recurrent word avoiding $f$ over $\Sigma_{\lambda(f)}$ is equivalent to a well-known morphic word.
In \Cref{sec:xyx},  we consider the avoidance of $xyx$-formulas and hybrid formulas.
In \Cref{sec:palin},  we consider the avoidance of patterns that are palindromes.

\section{Preliminaries}
The Zimin function associates to a pattern $p$ the pattern $Z(p)=pXp$ where $X$ is a variable that is not contained in $p$.
Notice that a recurrent word avoids $Z(p)$ if and only if it avoids $p$. In particular, $\lambda(p)=\lambda(Z(p))$.

We say that a formula $f$ is \emph{divisible} by a formula $f'$ if $f$ does not avoid $f'$, that is, there is a non-erasing morphism $h$
such that the image of any fragment of $f'$ by $h$ is a factor of a fragment of $f$. If $f$ is divisible by $f'$, then every word avoiding $f'$ also avoids $f$ and $\lambda(f')\ge \lambda(f)$.
Let $\Sigma_k=\acc{0,1,\ldots,k-1}$ denote the $k$-letter alphabet. We denote by $\Sigma_k^n$ the $k^n$ words of length $n$ over $\Sigma_k$.

The operation of \emph{splitting} a formula $f$ on a fragment $\phi$ consists in replacing $\phi$ by two fragments, namely the prefix and the suffix of length $|\phi|-1$ of $\phi$.
A formula $f$ is \emph{minimally avoidable} if splitting any fragment of $f$ gives an unavoidable formula. The set of every minimally avoidable formula with at most $n$ variables is called the $n$-avoidance basis.

The \emph{adjacency graph} $AG(f)$ of the formula $f$ is the bipartite graph such that
\begin{itemize}
\item for every variable $X$ of $f$, $AG(f)$ contains the two vertices $X_L$ and $X_R$,
\item for every (possibly equal) variables $X$ and $Y$, there is an edge between $X_L$ and $Y_R$ if and only if $XY$ is a factor of $f$.
\end{itemize}
We say that a set $S$ of variables of $f$ is \emph{free} if for all $X,Y\in S$, $X_L$ and $Y_R$ are in distinct connected components of $AG(f)$.
A formula $f$ is said to reduce to $f'$ if it is obtained by deleting all the variables of a free set from $f$, discarding any empty word fragment.
A formula is \emph{reducible} if there is a sequence of reductions to the empty formula. Finally, a \emph{locked} formula is a formula having no free set.
\begin{theorem}[\cite{BEM79}]\label{redequivunav}
A formula is unavoidable if and only if it is reducible.
\end{theorem}

Let us define here the following well-known pure morphic words. To specify a morphism $m:\Sigma_s\to\Sigma_e$, we use the notation $m=m(\texttt{0})/m(\texttt{1})/\cdots/m(s-1)$.
Assuming a morphism $m:\Sigma_s\to\Sigma_s$ is such that $m(\texttt{0})$ starts with \texttt{0}, the \emph{fixed point} of $m$ is the right infinite word $m^\omega(\texttt{0})$.

\begin{itemize}
\item$b_2$ is the fixed point of $\texttt{01}/\texttt{10}$.
\item$b_3$ is the fixed point of $\texttt{012}/\texttt{02}/\texttt{1}$.
\item$b_4$ is the fixed point of $\texttt{01}/\texttt{03}/\texttt{21}/\texttt{23}$.
\item$b_5$ is the fixed point of $\texttt{01}/\texttt{23}/\texttt{4}/\texttt{21}/\texttt{0}$
\end{itemize}
We also consider the morphic words
$v_3=M_1(b_5)$ and $w_3=M_2(b_5)$, where $M_1=\texttt{012}/\texttt{1}/\texttt{02}/\texttt{12}/\varepsilon$ and $M_2=\texttt{02}/\texttt{1}/\texttt{0}/\texttt{12}/\varepsilon$.
The languages of each of these words have been studied in the literature.
Let us first recall the following characterization of $b_3$, $v_3$, and $w_3$.
We say that two infinite words are \emph{equivalent} if they have the same set of factors.

\begin{theorem}[\cite{BO15}]\label{w_3}
\begin{itemize}
\item Every ternary square-free recurrent word avoiding \texttt{010} and \texttt{212} is equivalent to $b_3$.
\item Every ternary square-free recurrent word avoiding \texttt{010} and \texttt{020} is equivalent to $v_3$.
\item Every ternary square-free recurrent word avoiding \texttt{121} and \texttt{212} is equivalent to $w_3$.
\end{itemize}
\end{theorem}

Interestingly, these three words can be characterized in terms of a forbidden distance between consecutive occurrences of one letter.

\begin{theorem}\label{distance}{\ }
\begin{itemize}
\item Every ternary square-free recurrent word such that the distance between consecutive occurrences of \texttt{1} is not $3$ is equivalent to $b_3$.
\item Every ternary square-free recurrent word such that the distance between consecutive occurrences of \texttt{0} is not $2$ is equivalent to $v_3$.
\item Every ternary square-free recurrent word such that the distance between consecutive occurrences of \texttt{0} is not $4$ is equivalent to $w_3$.
\end{itemize}
\end{theorem}

The word $b_4$ is also known to avoid large families of formulas.
\begin{theorem}[\cite{BNT89}]\label{thm:locked}
Every locked formula is avoided by $b_4$.
\end{theorem}
\begin{theorem}[{\cite[Proposition 1.13]{Cassaigne1994}}]\label{thm:frag2}
If every fragment of an avoidable formula $f$ has length 2, then $b_4$ avoids~$f$.
\end{theorem}
\Cref{thm:frag2} will be extended to hybrid formulas, see \Cref{hybrid} in \Cref{sec:xyx}.

Let us give here a result that will be needed in various parts of the paper.
\begin{lemma}\label{sqf}
$ABA.ACA.ABCA.ACBA.ABCBA\preceq AA$.
\end{lemma}

\begin{proof}
Indeed, $Z^2(AA)=AABAACAABAA$ contains the occurrence $A\to A$, $B\to ABA$, $C\to ACA$ of $ABA.ACA.ABCA.ACBA.ABCBA$.
\qed
\end{proof}

Thus, if $w$ is a recurrent word that avoids a formula dividing $ABA.ACA.ABCA.ACBA.ABCBA$, then $w$ is square-free.

Recall that the repetition threshold $RT(n)$ is the smallest real number $\alpha$ such that there exists an infinite $a^+$-free word over $\Sigma_n$.
The proof of Dejean's conjecture established that $RT(2)=2$, $RT(3)=\tfrac75$, $RT(4)=\tfrac74$, and $RT(n)=\tfrac{n}{n-1}$ for every $n\ge5$.
An infinite $RT(n)^+$-free word over $\Sigma_n$ is called a Dejean word.

\section{Nice formulas}\label{sec:nice}
All the nice formulas considered so far in the literature are also $3$-avoidable.
This includes doubled patterns~\cite{O16}, circular formulas~\cite{circular}, the nice formulas in the $3$-avoidance basis~\cite{circular},
and the minimally nice ternary formulas in \Cref{mntf}~\cite{OchemRosenfeld2017}.

\begin{theorem}[\cite{circular,OchemRosenfeld2017}\label{n33}]
 Every nice formula with at most $3$ variables is $3$-avoidable.
\end{theorem}

We have a risky conjecture that would generalize both \Cref{n33} and the $3$-avoidability of doubled patterns.
\begin{conjecture}\label{conj:nice}
 Every nice formula is $3$-avoidable.
\end{conjecture}

\Cref{T_i} in \Cref{sec:xyx} shows that there exist infinitely many nice formulas with index $3$.
It means that Conjecture~\ref{conj:nice} would be best possible and it contrasts with the case of doubled patterns,
since we expect that there exist only finitely many doubled patterns with index $3$~\cite{O16,OchemPinlou}.
In this section, we make progress toward Conjecture~\ref{conj:nice} by proving that every nice formula is avoidable and we explain how to get an upper bound on the index of a given nice formula.

\subsection{The avoidability exponent}
Let us consider a useful tool in pattern avoidance that has been defined in~\cite{O16} and already used implicitly in~\cite{Ochem2004}.
The \emph{avoidability exponent} $AE(p)$ of a pattern~$p$ is the largest real $\alpha$ such that every $\alpha$-free word avoids $p$. We extend this definition to formulas.

Let us show that $AE(ABCBA.CBABC)=\tfrac43$. Suppose for contradiction that a $\tfrac43$-free word contains an occurrence $h$ of $ABCBA.CBABC$.
We write $y=|h(Y)|$ for every variable $Y$. The factor $h(ABCBA)$ is a repetition with period $|h(ABCB)|$. So we have $\frac{a+b+c+b+a}{a+b+c+b}<\tfrac43$.
This simplifies to $2a<2b+c$. Similarly, $CBABC$ gives $2c<a+2b$, $BAB$ gives $2b<a$, and $BCB$ gives $2b<c$. Summing up these four inequalities gives $2a+4b+2c<2a+4b+2c$, which is a contradiction.
On the other hand, the word \texttt{01234201567865876834201234} is $\paren{\tfrac43^+}$-free and contains the occurrence $A\to\texttt{01}$, $B\to\texttt{2}$, $C\to\texttt{34}$ of $ABCBA.CBABC$.

As a second example, we obtain that $AE(ABCDBACBD)=1.246266172\ldots$.
When we consider a repetition $uvu$ in an $\alpha$-free word, we derive that $\frac{|uvu|}{|uv|}<\alpha$, which gives $\beta|u|<|v|$ with $\alpha=1+\frac1{\beta+1}$. We consider an occurrence $h$ of the pattern.
The maximal repetitions in $ABCDBACBD$ are $ABCDBA$, $BCDB$, $BACB$, $CDBAC$, and $DBACBD$. They imply the following inequalities.
$$\left\{
\begin{array}{l}
 \beta a\le 2b+c+d\\
 \beta b\le c+d\\
 \beta b\le a+c\\
 \beta c\le a+b+d\\
 \beta d\le a+2b+c
\end{array}
\right.$$
We look for the smallest $\beta$ such that this system has no solution. Notice that $a$ and $d$ play symmetric roles. Thus, we can set $a=d$ and simplify the system.
$$\left\{
\begin{array}{l}
 \beta a\le a+2b+c\\
 \beta b\le a+c\\
 \beta c\le 2a+b
\end{array}
\right.$$
Then $\beta$ is the largest eigenvalue of the matrix $\cro{\begin{smallmatrix} 1 & 2 & 1\\ 1 & 0 & 1\\ 2 & 1 & 0\end{smallmatrix}}$ that corresponds to the latter system.
So $\beta=3.060647027\ldots$ is the largest root of the characteristic polynomial $x^3-x^2-5x-4$. Then $\alpha=1+\frac1{\beta+1}=1.246266172\ldots$

This matrix approach is a convenient trick to use when possible. It was used in particular for some doubled patterns such that every variable occurs exactly twice~\cite{O16}.
It may fail if the number of inequalities is strictly greater than the number of variables or if the formula contains a repetition $uvu$ such that $|u|\ge2$.
In any case, we can fix a rational value to $\beta$ and ask a computer algebra system whether the system of inequalities is solvable.
Then we can get arbitrarily good approximations of $\beta$ (and thus $\alpha$) by a dichotomy method.

Of course, the avoidability exponent is related to divisibility.
\begin{lemma}\label{ae}
If $f\preceq g$, then $AE(f)\le AE(g)$.
\end{lemma}

The avoidability exponent depends on the repetitions induced by $f$. We have $AE(f)=1$ for formulas such as $f=AB.BA.AC.CA.BC$ or $f=AB.BA.AC.BC.CDA.DCD$
that do not have enough repetitions. That is, for every $\varepsilon>0$, there exists a $(1+\varepsilon)$-free word that contains an occurrence of $f$.

Let us investigate formulas with non-trivial avoidability exponent, that is, $AE(f)>1$. To show that a nice formula has a non-trivial avoidability exponent (see \Cref{nice}),
we first introduce a notion of minimality for nice formulas similar to the notion of minimally avoidable for general formulas.
A nice formula $f$ is \emph{minimally nice} if there exists no nice formula $g$ such that $v(g)\le v(f)$ and $g\prec f$.
Alternatively, splitting a minimally nice formula on any of its fragments leads to a non-nice formula. The following property of every minimally nice formula is easy to derive.
If a variable $V$ appears as a prefix of a fragment $\phi$, then
\begin{itemize}
 \item $V$ is also a suffix of $\phi$ (since otherwise we can split on $\phi$ and obtain a nice formula),
 \item $\phi$ contains exactly two occurrences of $V$ (since otherwise we can remove the prefix letter $V$ from $\phi$ and obtain a nice formula),
 \item $V$ is neither a prefix nor a suffix of any fragment other than $\phi$ (since otherwise we can remove this prefix/suffix letter $V$ from the other fragment and obtain a nice formula),
 \item Every fragment other than $\phi$ contains at most one occurrence of $V$ (since otherwise we can remove the prefix letter $V$ from $\phi$ and obtain a nice formula).
\end{itemize}

\begin{lemma}\label{nice}
If $f$ is a nice formula, then $AE(f)\ge1+2^{1-v(f)}$.
\end{lemma}

\begin{proof}
Notice that $AE(AA)=2$ and $AE(ABA.BAB)=\tfrac32$, which settles the case $v(f)\le2$. Suppose that $f$ contradicts the lemma.
Since $1+2^{1-v(f)}$ is decreasing with $v(f)$, we can assume that $f$ is a minimally nice formula by Lemma~\ref{ae}.

Then there exists a $\paren{1+2^{1-v(f)}}$-free word $w$ containing an occurrence $h$ of~$f$. Let $X$ be a variable of $f$ such that $|h(X)|\ge|h(Y)|$ for every variable $Y$.
Thus, for every sequence $s$ of variables, $|h(s)|\le|s|\times|h(X)|$. Since $f$ is nice, $f$ contains a factor of the form $XzX$. By minimality, $z$ does not contain $X$, so that $v(z)\le v(f)-1$.

If $|z|\ge2^{v(z)}$, then $z$ contains a doubled pattern with at most $v(z)$ variables~\cite[Claim 3]{OchemPinlou}. This contradicts the minimality of $f$.

If $|z|\le2^{v(z)}-1$, then the exponent of $h(XzX)$ in $w$ is
$\frac{|h(XzX)|}{|h(Xz)|}=1+\frac{|h(X)|}{|h(Xz)|}\ge1+\frac{|h(X)|}{|Xz|\times|h(X)|}=1+\frac1{|Xz|}\ge1+\frac1{2^{v(z)}}\ge1+\frac1{2^{v(f)-1}}=1+2^{1-v(f)}$.
This contradicts that $w$ is $\paren{1+2^{1-v(f)}}$-free.
\qed
\end{proof}

The circular formulas studied in~\cite{circular} show that $AE(f)$ can be as low as $1+(v(f))^{-1}$.
Moreover, our example $AE(ABCDBACBD)=1.246266172\ldots$ shows that lower avoidability exponents exist among nice formulas with at least $4$ variables.
However, the bound in \Cref{nice} is probably very far from optimal.

We will describe below a method to construct infinite words avoiding a formula. This method can be applied if and only if the formula $f$ satisfies $AE(f)>1$.
So we are interested in characterizing the formulas $f$ such that $AE(f)>1$. By \Cref{ae,nice}, if $f$ is a formula such that there exists a nice formula $g$ satisfying $g\preceq f$, then $AE(f)>1$.
Now we prove that the converse also holds, which gives the following characterization.

\begin{theorem}
A formula $f$ satisfies $AE(f)>1$ if and only if there exists a nice formula $g$ such that $g\preceq f$.
\end{theorem}

\begin{proof}
What remains to prove is that for every formula $f$ that is not divisible by a nice formula and for every $\varepsilon>0$,
there exists an infinite $(1+\varepsilon)$-free word $w$ containing an occurrence of $f$, such that the size of the alphabet of $w$ only depends on $f$ and $\varepsilon$.

First, we consider the equivalent pattern $p$ obtained from $f$ by replacing every dot by a distinct variable that does not appear in $f$.
We will actually construct an occurrence of $p$. Then we construct a family $f_i$ of pseudo-formulas as follows.
We start with $f_0=p$. To obtain $f_{i+1}$ from $f_i$, we choose a variable that appears at most once in every fragment of $f_i$.
This variable is given the alias name $V_i$ and every occurrence of $V_i$ is replaced by a dot.
We say that $f_i$ is a pseudo-formula since we do not try to normalize $f_i$, that is, $f_i$ can contain consecutive dots
and $f_i$ can contain fragments that are factors of other fragments. However, we still have a notion of fragment for a pseudo-formula.
Since $f$ is not divisible by a nice formula, this process ends with the pseudo-formula $f_{v(p)}$ with no variable and $|p|$ consecutive dots.
The goal of this process is to obtain the ordering $V_0$, $V_1$, $\ldots$, $V_{v(p)-1}$ on the variables of $p$.

The image of every $V_i$ is a finite factor $w_i$ of a Dejean word over an alphabet of $\floor{\varepsilon^{-1}}+2$ letters, so that $w_i$ is $(1+\varepsilon)$-free.
The alphabets are disjoint: if $i\ne j$, then $w_i$ and $w_j$ have no common letter.
Finally, we define the length of $w_i$ as follows: $\abs{w_{v(p)-1}}=1$ and $\abs{w_i}=\floor{\varepsilon^{-1}}\times|p|\times|w_{i+1}|$ for every $i$ such that $0\le i\le v(p)-2$.
Let us show by contradiction that the constructed occurrence $h$ of $p$ is $(1+\varepsilon)$-free.
Consider a repetition $xyx$ of exponent at least $1+\varepsilon$ that is maximal, that is, which cannot be extended to a repetition with the same period and larger exponent.
Since every $w_i$ is $(1+\varepsilon)$-free and since two matching letters must come from distinct occurrences of the same variable, then $x=h(x')$ and $y=h(y')$
where $x'$ and $y'$ are factors of $p$. Our ordering of the variables of $p$ implies that $y'$ contains a variable $V_i$ such that $i<j$ for every variable $V_j$ in $x'$.
Thus, $|y|\ge\abs{w_i}=\floor{\varepsilon^{-1}}\times|p|\times|w_{i+1}|\ge\floor{\varepsilon^{-1}}\times\abs{x}$, which contradicts the fact that the exponent of $xyx$ is at least $1+\varepsilon$.

To obtain the infinite word $w$, we can insert our occurrence of $p$ into a bi-infinite $(1+\varepsilon)$-free word over an alphabet of $\floor{\varepsilon^{-1}}+2$ new letters.
So $w$ is an infinite $(1+\varepsilon)$-free word over an alphabet of $v(p)\paren{\floor{\varepsilon^{-1}}+2}+1$ letters which contains an occurrence of $f$.
\qed
\end{proof}

By \Cref{nice}, every nice formula is avoidable since it is avoided by a Dejean word over a sufficiently large alphabet.
Thus, if a formula is nice and minimally avoidable, then it is minimally nice. This is the case for every formula in the $3$-avoidance basis, except $AB.AC.BA.CA.CB$.
However, a minimally nice formula is not necessarily minimally avoidable. Indeed, we have shown~\cite{OchemRosenfeld2017} that the set of minimally nice ternary formulas
consists of the nice formulas in the $3$-avoidance basis, together with the minimally nice formulas in \Cref{mntf} that can be split to $AB.AC.BA.CA.CB$.

\begin{table}[htbp]
 \begin{itemize}
 \item $ABA.BCB.CAC$
 \item $ABCA.BCAB.CBAC$ and its reverse
 \item $ABCA.BAB.CAC$
 \item $ABCA.BAB.CBC$ and its reverse
 \item $ABCA.BAB.CBAC$ and its reverse
 \item $ABCBA.CABC$ and its reverse
 \item $ABCBA.CAC$
\end{itemize}
\caption{The minimally nice ternary formulas that are not minimally avoidable.}\label{mntf}
\end{table}

\subsection{Avoiding a nice formula}
Recall that a nice formula $f$ is such that $AE(f)>1$. We consider the smallest integer $s$ such that $RT(n)<AE(f)$.
Thus, every Dejean word over $\Sigma_s$ avoids~$f$, which already gives $\lambda(f)\le s$. Recall that a morphism is $q$-uniform if the image of every letter has length $q$.
Also, a uniform morphism $h:\Sigma^*_s\rightarrow\Sigma^*_e$ is \emph{synchronizing} if for any $a,b,c\in\Sigma_s$ and $v,w\in\Sigma^*_e$,
if $h(ab)=vh(c)w$, then either $v=\varepsilon$ and $a=c$ or $w=\varepsilon$ and $b=c$.
For increasing values of $q$, we look for a $q$-uniform morphism $h:\Sigma_s^*\to\Sigma_e^*$ such that $h(w)$ avoids $f$ for every $RT(s)^+$-free word $w\in\Sigma_s^{\ell}$,
where $\ell$ is given by Lemma~\ref{sync} below. Recall that a word is $(\beta^+,n)$-free if it contains no repetition with exponent strictly greater than $\beta$ and period at least $n$.

\begin{lemma}~\cite{Ochem2004}\label{sync}
Let $\alpha,\beta\in\mathbb{Q},\ 1<\alpha<\beta<2$ and $n\in\mathbb{N}^*$. Let $h:\Sigma^*_s\rightarrow\Sigma^*_e$ be a synchronizing $q$-uniform morphism (with $q\ge1$).
If $h(w)$ is $(\beta^+,n)$-free for every $\alpha^+$-free word $w$ such that $|w|<\max\paren{\frac{2\beta}{\beta-\alpha},\frac{2(q-1)(2\beta-1)}{q(\beta-1)}}$,
then $h(w)$ is $(\beta^+,n)$-free for every (finite or infinite) $\alpha^+$-free word~$w$.
\end{lemma}

Given such a candidate morphism $h$, we use Lemma~\ref{sync} to show that for every $RT(s)^+$-free word $w\in\Sigma_s^*$, the image $h(w)$ is $\paren{\beta^+,n}$-free.
The pair $\paren{\beta,n}$ is chosen such that $RT(s)<\beta<AE(f)$ and $n$ is the smallest possible for the corresponding $\beta$.
If $\beta<AE(f)$, then every occurrence $h$ of $f$ in a $\paren{\beta^+,t}$-free word is such that the length of the $h$-image of every variable of $f$ is upper bounded by a function of $n$ and $f$ only.
Thus, the $h$-image of every fragment of $f$ has bounded length and we can check that $f$ is avoided by inspecting a finite set of factors of words of the form $h(w)$.

\subsection{The number of fragments of a minimally avoidable formula}
Interestingly, the notion of (minimally) nice formula is helpful in proving the following.
\begin{theorem}
The only minimally avoidable formula with exactly one fragment is $AA$.
\end{theorem}

\begin{proof}
A formula with one fragment is a doubled pattern. Since it is minimally avoidable, it is a minimally nice formula.
By the properties of minimally nice formulas discussed above, the unique fragment of the formula is either $AA$ or is of the form $ApA$
such that $p$ does not contain the variable $A$. Thus, $p$ is a doubled pattern such that $p\prec ApA$, which contradicts that $ApA$ is minimally avoidable.
\qed
\end{proof}

By contrast, the family of \emph{two-birds} formulas, which consists of $ABA.BAB$, $ABCBA.CBABC$, $ABCDCBA.DCBABCD$, and so on, shows that
there exist infinitely many minimally avoidable formulas with exactly two fragments.
Every two-birds formula is nice.
Let us check that every two-birds formula $AB\cdots X\cdots BA.X\cdots A\cdots X$ is minimally avoidable. Since the two fragments play symmetric roles, it is sufficient to split on the first fragment.
We obtain the formula $AB\cdots X\cdots B.B\cdots X\cdots BA.X\cdots A\cdots X$ which divides the pattern $B\cdots X\cdots BAB\cdots X\cdots B=Z(B\cdots X\cdots B)$.
This pattern is equivalent to $B\cdots X\cdots B$, which is unavoidable. Thus, every two-birds formula is indeed minimally avoidable.

Concerning the index of two-birds formulas, we have seen that $\lambda(ABA.BAB)=3$ and $\lambda(ABCBA.CBABC)=2$ \cite{circular}. Computer experiments suggest that larger two-birds formulas are easier to avoid.
\begin{conjecture}\label{conj:2birds}
 Every two-birds formula with at least $3$ variables is $2$-avoidable.
\end{conjecture}

\section{Characterization of some famous morphic words}\label{sec:charac}
Our next result give characterizations of $w_3$, up to renaming, that use just one formula. Then we give similar characterizations of $b_3$ and $b_2$.
Let $\sigma=\texttt{1}/\texttt{2}/\texttt{0}$ be the morphism that cyclically permutes $\Sigma_3$.

\begin{theorem}\label{transitive}
Let $f$ be a ternary formula such that $ABA.BCB.ACA\preceq f\preceq ABA.ABCBA.ACA.ACB.BCA$. Every ternary recurrent word avoiding $f$ is equivalent to $w_3$, $\sigma(w_3)$, or $\sigma^2(w_3)$.
\end{theorem}

\begin{proof}
Using Cassaigne's algorithm~\cite{cassaignealgo}, we have checked that $w_3$ avoids $ABA.BCB.ACA$. By divisibility, $w_3$ avoids~$f$.

Let $w$ be a ternary recurrent word avoiding $f$. By Lemma~\ref{sqf}, $w$ is square-free.

Let $v=\texttt{210201202101201021}$. A computer check shows that no infinite ternary word avoids $ABA.ABCBA.ACA.ACB.BCA$, squares, $v$, $\sigma(v)$, and $\sigma^2(v)$.
So, without loss of generality, $w$ contains $v$. If $w$ contains \texttt{121}, then $w$ contains the occurrence $A\to\texttt{1}$, $B\to\texttt{2}$, $C\to\texttt{0}$ of $ABA.ACA.ABCA.ACBA.ABCBA$.
Similarly, if $w$ contains \texttt{212}, then $w$ contains the occurrence $A\to\texttt{2}$, $B\to\texttt{1}$, $C\to\texttt{0}$ of $ABA.ACA.ABCA.ACBA.ABCBA$.
Thus, $w$ avoids squares, \texttt{121}, and \texttt{212}. By \Cref{w_3}, $w$ is equivalent to $w_3$

By symmetry, every ternary recurrent word avoiding $f$ is equivalent to $w_3$, $\sigma(w_3)$, or $\sigma^2(w_3)$.
\qed
\end{proof}

\begin{theorem}\label{thm:b3} 
Let $f$ be such that
\begin{itemize}
 \item $ABCA.ABA.ACA\preceq f\preceq ABCA.ABA.ACA.ACB.CBA$,
 \item $ABCA.ABA.BCB.AC\preceq f\preceq ABCA.ABA.ABCBA.ACB$, or 
 \item $ABCA.ABA.BCB.CBA\preceq f\preceq ABCA.ABA.ABCBA.ACB$.
\end{itemize}
Every ternary recurrent word avoiding $f$ is equivalent to $b_3$, $\sigma(b_3)$, or $\sigma^2(b_3)$.
\end{theorem}

\begin{proof}
Using Cassaigne's algorithm~\cite{cassaignealgo}, we have checked that $b_3$ avoids $ABCA.ABA.ACA$, $ABCA.ABA.BCB.AC$, and $ABCA.ABA.BCB.CBA$.
By divisibility, $b_3$ avoids~$f$. Let $w$ be a ternary recurrent word avoiding $f$. By Lemma~\ref{sqf}, $w$ is square-free.

Let $v=\texttt{20210121020120}$. A computer check shows that no infinite ternary word avoids $ABCA.ABA.ACA.ACB.CBA$ (resp. $ABCA.ABA.ABCBA.ACB$), squares, $v$, $\sigma(v)$, and $\sigma^2(v)$.

So, without loss of generality, $w$ contains $v$. If $w$ contains \texttt{010}, then $w$ contains the occurrence $A\to\texttt{0}$, $B\to\texttt{1}$, $C\to\texttt{2}$ of $ABA.ACA.ABCA.ACBA.ABCBA$.
Similarly, if $w$ contains \texttt{212}, then $w$ contains the occurrence $A\to\texttt{2}$, $B\to\texttt{1}$, $C\to\texttt{0}$ of $ABA.ACA.ABCA.ACBA.ABCBA$.
Thus, $w$ avoids squares, \texttt{010}, and \texttt{212}. By \Cref{w_3}, $w$ is equivalent to $b_3$.

By symmetry, every ternary recurrent word avoiding $f$ is equivalent to $b_3$, $\sigma(b_3)$, or $\sigma^2(b_3)$.
\qed
\end{proof}

Notice that \Cref{thm:b3} is a complement to~\cite[Theorem 2]{OchemRosenfeld2017} in which we gave a disjoint set of formulas with the same property.
The difference between \Cref{thm:b3} and~\cite[Theorem 2]{OchemRosenfeld2017} is that a different occurrence of $f$ shows that $f$ divides $Z^n(AA)$.  

\begin{theorem}\label{tm2}
Let $f_h=AABCAA.BCB$, $f_e=AABCAAB.AABCAB.AABCB$, and let $f$ be such that $f_h\preceq f\preceq f_e$. Every binary recurrent word avoiding $f$ is equivalent to $b_2$.
\end{theorem}

\begin{proof}
Using Cassaigne's algorithm~\cite{cassaignealgo}, we have checked that $b_2$ avoids $f_h$. First, $f_e\preceq AAA$ because $Z(AAA)=AAABAAA$ contains the occurrence $A\to A$, $B\to A$, $C\to B$ of $f_e$.
Second, $f_e\preceq ABABA$ because $Z(ABABA)=ABABACABABA$ contains the occurrence $A\to AB$, $B\to A$, $C\to C$ of $f_e$.

Thus, every recurrent word avoiding $f_e$ also avoids $AAA$ and $ABABA$, which means that it is overlap-free. Finally, it is well-known that every binary recurrent word that is overlap-free is equivalent to $b_2$.
\qed
\end{proof}

\section{$xyx$-formulas}\label{sec:xyx}
Recall that every fragment of an $xyx$-formula is of the form $XYX$. We associate to an $xyx$-formula $F$ the directed graph $\overrightarrow{G}$ such that every variable corresponds to a vertex
and $\overrightarrow{G}$ contains the arc $\overrightarrow{XY}$ if and only if $F$ contains the fragment $XYX$. We will also denote by $G$ the underlying simple graph of $\overrightarrow{G}$.

\begin{lemma}\label{homo}
Let $F_1$ and $F_2$ be $xyx$-formulas associated to $\overrightarrow{G_1}$ and $\overrightarrow{G_2}$. If there exists a homomorphism $\overrightarrow{G_1}\to\overrightarrow{G_2}$, then $F_1\preceq F_2$.
\end{lemma}

\begin{proof}
Since both digraph homomorphism and formula divisibility are transitive relations, we only need to consider the following two cases.
If $G_1$ is a subgraph of $G_2$, then $F_1$ is obtained from $F_2$ by removing some fragments. So every occurrence of $F_2$ is also an occurrence of $F_1$ and thus $F_1\preceq F_2$.
If $G_2$ is obtained from $G_1$ by identifying the vertices $u$ and $v$, then $F_2$ is obtained from $F_1$ by identifying the variables $U$ and $V$.
So every occurrence of $F_2$ is also an occurrence of $F_1$ and thus $F_1\preceq F_2$.
\qed
\end{proof}

For every $i$, let $T_i$ be the $xyx$-formula corresponding to the directed circuit $\overrightarrow{C_i}$ of length $i$, that is, $T_1=AAA$, $T_2=ABA.BAB$, $T_3=ABA.BCB.CAC$, $T_4=ABA.BCB.CDC.DAD$, and so on.
More formally, $T_i$ is the formula with $i$ variables $A_0$, $\ldots$, $A_{i-1}$ which contains the $i$ fragments of length three of the form $A_jA_{j+1}A_j$ such that the indices are taken modulo $i$.
Notice that $T_i$ is a nice formula.
\begin{theorem}\label{T_i}
For every $i\ge2$, $\lambda(T_i)=3$
\end{theorem}

\begin{proof}
We use Lemma~\ref{sync} to show that the image of every $(7/4^+)$-free word over $\Sigma_4$ by the following $58$-uniform morphism is $(3/2,3)$-free.
$$
\begin{array}{ll}
\texttt{0}\to&\texttt{0012211002201021120022100112201002112001022011002211201022}\\
\texttt{1}\to&\texttt{0012210022010211220010221120011022010021122011002211201022}\\
\texttt{2}\to&\texttt{0011221002201021122001102201002112001022110012200211201022}\\
\texttt{3}\to&\texttt{0011221002201021120011022010021122001022110012200211201022}\\
\end{array}
$$
In these words, the factor \texttt{010} is the only occurrence $m$ of $ABA$ such that $|m(A)|\ge|m(B)|$. This implies that these ternary words avoid $T_i$ for every $i\ge1$, so that $\lambda(T_i)\le3$.
%
%

To show that $T_i$ is not $2$-avoidable, we consider the $xyx$-formula $H=ABA.BAB.ACA.CBC$ associated to the directed graph $\overrightarrow{D_3}$ on 3 vertices and 4 arcs
that contains a circuit of length 2 and a circuit of length 3. Standard backtracking shows that $\lambda(H)>2$, and even the stronger result that $\lambda(ABAB.ACA.CAC.BCB.CBC)>2$.

For every $i\ge2$, the circuit $\overrightarrow{C_i}$ admits a homomorphism to $\overrightarrow{D_3}$. By Lemma~\ref{homo}, this means that $T_i\preceq H$, which implies that $\lambda(T_i)\ge\lambda(H)\ge3$.
\qed
\end{proof}

\begin{theorem}\label{ti_b4}
For every $i\ge1$, $b_4$ avoids $T_i$.
\end{theorem}

\begin{proof}
Suppose for contradiction that there exist $i$ and $n$ such that $m^n(\texttt{0})$ contains an occurrence $h$ of $T_i$. Further assume that $n$ is minimal.
Notice that in $b_4$, every even (resp. odd) letter appears only at even (resp. odd) positions. Thus, for every fragment $XYX$ of $T_i$, the period $|h(XY)|$ of the repetition $h(XYX)$ must be even.
This implies that $|h(X)|$ and $|h(Y)|$ have the same parity. By contagion, the lengths of the images of all the variables of $T_i$ have the same parity. Now we proceed to a case analysis.
\begin{itemize}
 \item Every $|h(X)|$ is even.
 \begin{itemize}
  \item Every $h(X)$ starts with $\texttt{0}$ or $\texttt{2}$. By taking the pre-image by $m$ of every $h(X)$, we obtain an occurrence of $T_i$ that is contained in $m^{n-1}(\texttt{0})$.
  This contradicts the minimality of $n$.
  \item Every $h(X)$ starts with $\texttt{1}$ or $\texttt{3}$. Notice that in $b_4$, the letter $\texttt{1}$ (resp. $\texttt{3}$) is in position $1\pmod{4}$ (resp. $3\pmod{4}$).
  $m^n(\texttt{0})$ contains the occurrence $h'$ of $T_i$ such that $h'(X)$ is obtained from $h(X)$ by adding to the rigth the letter $\texttt{1}$ or $\texttt{3}$
  depending on its position modulo $4$ and by removing the first letter. 
  Since is also contained in $m^n(\texttt{0})$ and every $h'(X)$ starts with $\texttt{0}$ or $\texttt{2}$, $h'$ satisfies the previous subcase.
 \end{itemize}
  \item Every $|h(X)|$ is odd. It is not hard to check that every factor $uvu$ in $b_4$ with $|v|=1$ satisfies $v\in\acc{\texttt{1},\texttt{3}}$ and $u\in\acc{\texttt{0},\texttt{2}}$.
  So $|h(X)|\ge3$ for every variable $X$ of $T_i$. Let $X_1,\cdots, X_i$ be the variables of $T_i$. Up to a shift of indices, we can assume that $j$ and the first and last letters of $h(X_j)$ have the same parity.
  We construct the occurrence $h'$ of $T_i$ as follows. If $j$ is odd, then $h'(X_j)$ is obtained by removing the first letter of $h(X_j)$.
  If $j$ is even, then $h'(X_j)$ is obtained by adding to the right the letter $\texttt{1}$ or $\texttt{3}$ depending on its position modulo $4$.
  Since $h'$ is also contained in $m^n(\texttt{0})$ and every $|h'(X)|$ is even, $h'$ satisfies the previous case.
\end{itemize}
\qed
\end{proof}

Our next result generalizes \Cref{ti_b4,thm:frag2}. Recall that every fragment of a hybrid formula has length 2 or is of the form $XYX$.
\begin{theorem}\label{hybrid}
Every avoidable hybrid formula is avoided by $b_4$.
\end{theorem}

\begin{proof}
Let $f$ be a hybrid formula. If $f$ contains a locked formula or a formula $T_i$, then $b_4$ avoids $f$ by \Cref{thm:locked,ti_b4}.
If $f$ contains neither a locked formula nor a formula $T_i$, then we show that $f$ is unavoidable.
By induction and by \cref{redequivunav} it is sufficient to show that $f$ is reducible to a hybrid formula containing neither a locked formula nor a formula $T_i$.
Since $f$ is not locked, $f$ contains a free set of variables and thus $f$ has a free singleton $\acc{X}$. If $f$ contains a fragment $YXY$, then $\acc{Y}$ is also a free singleton of $f$.
Using this argument iteratively, we end up with a free singleton $\acc{Z}$ such that $f$ contains no fragment $TZT$, since $f$ contains no formula $T_i$.

So we can assume that $f$ contains a free singleton $\acc{Z}$ and no fragment $TZT$.
Thus, deleting every occurrence of $Z$ from $f$ gives an hybrid sub-formula containing neither a locked formula nor a formula $T_i$. By induction, $f$ is unavoidable.
\qed
\end{proof}

So the index of an avoidable $xyx$-formula is at most $4$ and we have seen examples of $xyx$-formulas with index $3$ in \Cref{transitive,T_i}.
The next results give an $xyx$-formula with index $4$ and an $xyx$-formula with index $2$ that is not divisible by $AAA$.

\begin{theorem}\label{c5}
$\lambda(ABA.BCB.DCD.DED.AEA)=4$.
\end{theorem}

\begin{proof}
By \Cref{hybrid}, $ABA.BCB.DCD.DED.AEA$ is $4$-avoidable. Notice that $ABA.BCB.DCD.DED.AEA\preceq ABA.BCB.ACA$ via the homomorphism $A\to A$, $B\to B$, $C\to C$, $D\to B$, $E\to C$.
Moreover, $w_3$ contains the occurrence $A\to\texttt{0}$, $B\to\texttt{1}$, $C\to\texttt{02}$, $D\to\texttt{01}$, $E\to\texttt{2}$ of $ABA.BCB.DCD.DED.AEA$. By \Cref{transitive}, the formula is not $3$-avoidable.
\qed
\end{proof}

\begin{theorem}\label{k4}
The fixed point of $\texttt{001}/\texttt{011}$ avoids the $xyx$-formula associated to the directed graph on $4$ vertices with all the $12$ arcs.
\end{theorem}

\begin{proof}
We use again Cassaigne's algorithm.
\qed
\end{proof}

%

\section{Palindrome patterns}\label{sec:palin}
Mikhailova~\cite{M13} has considered the index of an avoidable pattern that is a palindrome and proved that it is at most $16$.
She actually constructed a morphic word over $\Sigma_{16}$ that avoids every avoidable palindrome pattern.

We make a distinction between the largest index $\mathcal{P}_w$ of an avoidable palindrome pattern and the smallest alphabet size $\mathcal{P}_s$ allowing an infinite word avoiding every avoidable palindrome pattern.
We obtained~\cite{OchemRosenfeld2017} the lower bound $\lambda(ABCADACBA)=\lambda(ABCA.ACBA)=4$, so that $4\le\mathcal{P}_w\le\mathcal{P}_s\le16$.

The following result is a slight improvement to $\lambda(ABCA.ACBA)=4$ that is not related to palindromes.
\begin{theorem}\label{improvement}
$\lambda(ABCA.ACBA.ABCBA)=4$.
\end{theorem}

\begin{proof}
By Lemma~\ref{sqf}, every recurrent word avoiding $ABCA.ACBA.ABCBA$ is square-free.
A computer check shows that no infinite ternary square-free word avoids the occurrences $h$ of $ABCA.ACBA.ABCBA$ such that $|h(A)|=1$, $|h(B)|\le 2$, and $|h(C)|\le3$.
\qed
\end{proof}

Let us give necessary conditions on a palindrome pattern $P$ so that $5\le\lambda(P)\le16$.

\begin{enumerate}
 \item The length of $P$ is odd and the central variable of $P$ is isolated. Indeed, otherwise $P$ would be a doubled pattern and thus $3$-avoidable~\cite{O16}.
 \item No variable of $P$ appears both at an even and an odd position. Indeed, if $P$ had a variable that appears both at an even and an odd position,
  then $P$ would be divisible by a formula in the family $AA$, $ABCA.ACBA$, $ABCDEA.AEDCBA$, $ABCDEFGA.AGFEDCBA$, \dots
  Such formulas (with an odd number of variables) are locked and thus are avoided by $b_4$ by \Cref{thm:locked}. So $P$ would be $4$-avoidable.
\end{enumerate}

We have found three patterns/formulas satisfying these conditions (see \Cref{palin}), but they seem to be 2-avoidable.
We use again Cassaigne's algorithm with simple pure morphic words to ensure that they are 4-avoidable. Let $z_3$ be the fixed point of $\texttt{01}/\texttt{2}/\texttt{20}$.

\begin{theorem}\label{palin}{\ }
\begin{enumerate}
 \item $ADBDCDAD.DADCDBDA$ is avoided by $b_4$.
 \item $ABCDADC.CDADCBA$ is avoided by $z_3$.
 \item $ABACDBAC.CABDCABA$ is avoided by $z_3$ and $b_4$. 
\end{enumerate}
\end{theorem}

\section{Discussion}
Let us briefly mention the things that we have attempted to do in this paper, without success.
\begin{itemize}
 \item Improve the bound in \Cref{nice}.
 \item Improve \Cref{k4} by showing that some $xyx$-formula on $4$ variables and fewer fragments is $2$-avoidable.
 \item Show that the $xyx$-formula associated to the transitive tournament on $5$ vertices is $2$-avoidable.
\end{itemize}

\end{document}